\documentclass{siamltex}
\usepackage{amsfonts,amsmath,amssymb}
\usepackage{mathrsfs,mathtools}
\usepackage{enumerate}
\usepackage{xspace,MACROS/mydef} 
\usepackage{hyperref}
\usepackage{esint}
\usepackage{graphicx}
\DeclareMathAlphabet{\mathpzc}{OT1}{pzc}{m}{it}

\newtheorem{remark}[theorem]{Remark}
\numberwithin{equation}{section}

\title{Some applications of weighted norm inequalities to the error analysis of PDE constrained optimization problems\thanks{HA has been supported in part by the NSF grant DMS-1521590. EO has been supported in part by CONICYT through FONDECYT project 3160201. AJS has been supported in part by NSF grant DMS-1418784.}}

\author{Harbir Antil\thanks{Department of Mathematical Sciences, George Mason University, Fairfax, VA 22030, USA. \texttt{hantil@gmu.edu}},
\and
Enrique Ot\'arola\thanks{Departamento de Matem\'atica, Universidad T\'ecnica Federico Santa Mar\'ia, Valpara\'iso, Chile.
\texttt{enrique.otarola@usm.cl}}
\and
Abner J.~Salgado\thanks{Department of Mathematics, University of Tennessee, Knoxville, TN 37996, USA. \texttt{asalgad1@utk.edu}}}

\date{Draft version of \today.}

\begin{document}

\maketitle

\begin{abstract}
The purpose of this work is to illustrate how the theory of Muckenhoupt weights, Muckenhoupt weighted Sobolev spaces and the corresponding weighted norm inequalities can be used in the analysis and discretization of PDE constrained optimization problems. We consider: a linear quadratic constrained optimization problem where the state solves a nonuniformly elliptic equation; a problem where the cost involves pointwise observations of the state and one where the state has singular sources, \eg point masses. For all three examples we propose and analyze numerical schemes and provide error estimates in two and three dimensions. While some of these problems might have been considered before in the literature, our approach allows for a simpler, Hilbert space-based, analysis and discretization and further generalizations.
\end{abstract}

%
\section{Introduction}
\label{sec:introduccion}

The purpose of this work is to show how the theory of Muckenhoupt weights, Muckenhoupt weighted Sobolev spaces and weighted norm inequalities can be applied to analyze PDE constrained optimization problems, and their discretizations. These tools have already been shown to be essential in the analysis and discretization of problems constrained by equations involving fractional derivatives both in space and time \cite{AO,AOS}, and here we extend their use to a new class of problems.

We consider three illustrative examples. While some of them have been considered before, the techniques that we present are new and we believe they provide simpler arguments and allow for further generalizations. To describe them let $\Omega$ be an open and bounded polytopal domain of $\R^n$ ($n \in \{2,3\}$) with Lipschitz boundary $\partial \Omega$. We will be concerned with the following problems:
\begin{enumerate}[$\bullet$]
  \item 
  \textbf{Optimization with nonuniformly elliptic equations.} Let $\omega$ be a weight, that is, a positive and locally integrable function and $y_d \in L^2(\omega,\Omega)$. Given a regularization parameter $\lambda>0$, we define the cost functional
  \begin{equation}
  \label{eq:defofJA}
    J_\A(y,u) = \frac12 \| y - y_d \|_{L^2(\omega,\Omega)}^2 + \frac\lambda2 \| u \|_{L^2(\omega^{-1},\Omega)}^2.
  \end{equation}
  We are then interested in finding $\min J_\A$ subject to the \emph{nonuniformly elliptic problem}
  \begin{equation}
  \label{eq:defofPDEA}
    -\DIV( \A \GRAD y ) = u \text{ in } \Omega, \qquad y = 0 \text{ on } \partial\Omega,
  \end{equation}
  and the control constraints
  \begin{equation}
  \label{eq:cc}
    u \in \bbU_{\A},
  \end{equation}
  where $\bbU_{\A}$ is a  nonempty, closed and convex subset of $L^2(\omega^{-1},\Omega)$. The main source of difficulty and originality here is that the matrix $\A$ is not uniformly elliptic, but rather satisfies
  \begin{equation}
  \label{eq:nonunifellip}
    \omega(x) |\xi|^2 \lesssim \xi^\intercal \cdot \A(x) \cdot \xi \lesssim \omega(x) |\xi|^2 \quad \forall \xi \in \R^n
  \end{equation}
  for almost every $x\in \Omega$. Since we allow the weight to vanish or blow up, this nonstandard ellipticity condition must be treated with the right functional setting.
  
  Problems such as \eqref{eq:defofPDEA} arise when applying the so-called Caffarelli-Silvestre extension for fractional diffusion \cite{AO,AOS,CS:07,NOS,NOS2}, when dealing with boundary controllability of parabolic and hyperbolic degenerate equations \cite{MR2373460,MR3171770,MR3227458} and in the numerical approximation of elliptic problems involving measures \cite{MR3264365,NOS2}. In addition, invoking Rubio de Francia's extrapolation theorem \cite[Theorem 7.8]{MR1800316} one can argue that this is a quite general PDE constrained optimization problem with an elliptic equation as state constraint, since \emph{there is no $L^p$, only $L^2$ with weights}.

  \item 
  \textbf{Optimization with point observations.} Let $\calZ \subset \Omega$ with $\# \calZ < \infty$. Given a set of prescribed values $\{ y_z\}_{z \in \mathcal{Z}}$, a regularization parameter $\lambda >0$, and the cost functional
  \begin{equation}
  \label{eq:defofJp}
    J_\calZ(y,u) = \frac12 \sum_{z \in \calZ} | y(z) - y_z |^2 + \frac\lambda2 \| u \|_{L^2(\Omega)}^2,
  \end{equation}
  the problem under consideration reads as follows: Find $\min J_\calZ$ subject to 
  \begin{equation}
  \label{eq:defofPDEp}
    -\LAP y = u \text{ in } \Omega, \qquad y = 0 \text{ on } \partial\Omega,
  \end{equation}
  and the control constraints
  \begin{equation}
  \label{eq:cc2}
  u \in \U_\calZ,
  \end{equation}
  where $\U_\calZ$ is a nonempty, closed and convex subset of $L^2(\Omega)$. In contrast to standard elliptic PDE constrained optimization problems, the cost functional \eqref{eq:defofJp} involves point evaluations of the state. Note that these evaluations are not required for the state equation \eqref{eq:defofPDEp} to be well posed. Additional assumptions must be made for this to make sense and, as will be seen below, the point evaluations of the state $y$ in \eqref{eq:defofJp} lead to a subtle formulation of the adjoint problem.
  
  Problem \eqref{eq:defofJp}--\eqref{eq:cc2} finds relevance in numerous applications where the observations are carried out at specific locations. For instance, in the so-called calibration problem with American options \cite{MR2137495}, in the optimal control of selective cooling of steel \cite{MR1844451}, in the active control of sound \cite{MR2086168,ACS} and in the active control of vibrations \cite{ACV,MR2525606}; see also \cite{BrettElliott,BEHL,MR3150173,MR2536481,MR2193509} for other applications.
  
  The point observation terms in the cost \eqref{eq:defofJp}, tend to enforce the state $y$ to have the fixed value $y_z$ at the point $z$. Consequently, \eqref{eq:defofJp}--\eqref{eq:cc2} can be understood as a penalty version of a PDE constrained optimization problem where the state is constrained at a collection of points. We refer the reader to \cite[Section 3.1]{BrettElliott} for a precise description of this connection and to \cite{MR3071172} for the analysis and discretization of an optimal control problem with state constraints at a finite number of points.

  Despite its practical importance, to the best of our knowledge, there are only two references where the approximation of \eqref{eq:defofJp}--\eqref{eq:cc2} is addressed: \cite{BrettElliott} and \cite{MR3449612}. In both works the key observation, and main source of difficulty, is that the adjoint state for this problem is only in $W_0^{1,r}(\Omega)$ with $r \in (\tfrac{2n}{n+2}, \tfrac{n}{n-1})$. With this functional setting, the authors of \cite{BrettElliott} propose a fully discrete scheme which discretizes the control explicitly using \emph{piecewise linear elements}. For $n=2$, the authors obtain a $\calO(h)$ rate of convergence for the optimal control in the $L^2$-norm provided the control and the state are discretized using meshes of size $\calO(h^2)$ and $\calO(h)$, respectively; see \cite[Theorem~5.1]{BrettElliott}. This condition immediately poses two challenges for implementation: First, it requires to keep track of the state and control on different meshes. Second, some sort of interpolation and projection between these meshes needs to be realized. In addition, the number of unknowns for the control is significantly higher, thus leading to a slow optimization solver. The authors of \cite{BrettElliott} were unable to extend these results to $n=3$. Using the so-called variational discretization approach \cite{Hinze:05}, the control is implicitly discretized and the authors were able to prove that the control converges with rates $\calO(h)$ for $n=2$ and $\calO(h^{\frac{1}{2}-\epsilon})$ for $n=3$. In a similar fashion, the authors of \cite{MR3449612} use the variational discretization concept to obtain an implicit discretization of the control and deduce rates of convergence of $\calO(h)$ and $\calO(h^{\frac{1}{2}})$ for $n=2$ and $n=3$, respectively. A residual-type a posteriori error estimator is introduced, and its reliability is proven. However, there is no analysis of the efficiency of the estimator.

  In Section~\ref{sec:points} below we introduce a fully discrete scheme where we discretize the control with piecewise constants; this leads to a smaller number of degrees of freedom for the control in comparison to the approach of \cite{BrettElliott}. We circumvent the difficulties associated with the adjoint state by working in a weighted $H^1$-space and prove near optimal rates of convergence for the optimal control: $\calO(h|\log h|)$ for $n=2$ and $\calO(h^\frac{1}{2}|\log h|^2)$ for $n=3$, respectively. In addition, we provide pointwise error estimates for the approximation of the state: $\calO(h|\log h|)$ for $n=2$ and $\calO(h^{\frac{1}{2}}|\log h|^2)$ for $n=3$.
  
  \item \textbf{Optimization with singular sources.} Let $\calD \subset \Omega$ be linearly ordered and  with cardinality $l= \# \calD < \infty$. Given a desired state $y_d \in L^2(\Omega)$ and a regularization parameter $\lambda>0$, we define the cost functional
  \begin{equation}
  \label{eq:defofJd}
    J_\delta(y,\mathbf{u}) = \frac12 \| y - y_d \|_{L^2(\Omega)}^2 + \frac\lambda2 \| \mathbf{u} \|_{\R^l}^2.
  \end{equation}
  We shall be concerned with the following problem: Find $\min J_\delta$ subject to
  \begin{equation}
  \label{eq:defofPDEd}
    -\LAP y = \sum_{z \in \calD} u_z \delta_z \text{ in } \Omega, \qquad y = 0 \text{ on } \partial\Omega,
  \end{equation}
  where $\delta_z$ is the Dirac delta at the point $z$ and 
  \begin{equation}
  \label{eq:cc3}
   \mathbf{u} = \{u_z\}_{z \in \calD} \in \bbU_\delta,
  \end{equation}
  where $\bbU_\delta \subset \R^l$ with $\bbU_\delta$, again, nonempty, closed and convex. Notice that since, for $n>1$, $\delta_z \notin H^{-1}(\Omega)$, the solution $y$ to \eqref{eq:defofPDEd} does not belong to $H^1(\Omega)$. Consequently, the analysis of the finite element method applied to such a problem is not standard \cite{Casas:85,NOS2,Scott:73}. We rely on the weighted Sobolev space setting described and analyzed in \cite[Section 7.2]{NOS2}.
  
  The state \eqref{eq:defofPDEd}, in a sense, is dual to the adjoint equation for \eqref{eq:defofJp}--\eqref{eq:defofPDEp}, where the adjoint equation has Dirac deltas on the right hand side. The optimization problem \eqref{eq:defofJd}--\eqref{eq:defofPDEd} is of relevance in applications where one can specify a control at a finitely many pre-specified points. For instance, references \cite{MR2086168,ACS} discuss applications within the context of the active control of sound and \cite{ACV,MR2718176,MR2525606} in the active control of vibrations; see also \cite{MR3268059,MR3150173,MR3116646}.
  
 An analysis of problem \eqref{eq:defofJd}--\eqref{eq:cc3} is presented in \cite{MR3225501}, where the authors use the variational discretization concept to derive error estimates. They show that the control converges with a rate of $\calO(h)$ and $\calO(h^{1/2})$ in two and three dimensions, respectively. Their technique is based on the fact that the state belongs to $W_0^{1,r}(\Omega)$ with $r \in (\tfrac{2n}{n+2}, \tfrac{n}{n-1})$. In addition, under the assumption that $y_d \in L^\infty(\Omega)$ they improve their results and obtain, up to logarithmic factors, rates of $\calO(h^2)$ and $\calO(h)$. Finally, we mention that \cite{MR2974716, MR3072225} study a PDE constrained optimization problem without control constraints, but where the controls is a regular Borel measure.
 
 In Section~\ref{sec:delta} we present a fully discrete scheme for which we provide rates of convergence for the optimal control: $\calO(h^{2-\epsilon})$ in two dimensions and $\calO(h^{1-\epsilon})$ in three dimensions, where $\epsilon>0$. We also present rates of convergence for the approximation error in the state variable.
\end{enumerate}

Before we embark in further discussions, we must remark that while the introduction of a weight as a technical instrument does not seem to be completely new, the techniques that we use and the range of problems that we can tackle is. For instance, for integro-differential equations  where the kernel $g$ is weakly singular, the authors of \cite{MR1306580} study the well-posedness of the problem in the weighted $L^2(g,(-r,0)),$ space. Numerical approximations for this problem with the same functional setting were considered in \cite{MR1135762} where convergence is shown but no rates are obtained. These ideas were extended to neutral delay-differential equations in \cite{MR3064275,MR2018123} where a weight is introduced in order to renorm the state space and obtain dissipativity of the underlying operator. In all these works, however, the weight is essentially assumed to be smooth and monotone except at the origin where it has an integrable singularity \cite{MR1306580,MR1135762} or at a finite number of points where it is allowed to have jump discontinuities \cite{MR3064275,MR2018123}. All these properties are used to obtain the aforementioned results. In contrast, our approach hinges \emph{only} on the fact that the introduced weights belong to the Muckenhoupt class $A_2$ (see Definition~\ref{def:defofAp} below) and the pertinent facts from real and harmonic analysis and approximation theory that follow from this definition. Additionally we obtain convergence rates which, up to logarithmic factors, are optimal with respect to regularity. Finally we must point out that the class of problems we study is quite different from those considered in the references given above.

Our presentation will be organized as follows. Notation and general considerations will be introduced in Section~\ref{sec:notation}. Section~\ref{sec:A} presents the analysis and discretization of problem \eqref{eq:defofJA}--\eqref{eq:cc}. Problem \eqref{eq:defofJp}--\eqref{eq:cc2} is studied in Section~\ref{sec:points}. The analysis of problem \eqref{eq:defofJd}--\eqref{eq:cc3} is presented in Section~\ref{sec:delta}. Finally, in Section~\ref{sec:NumExp}, we illustrate our theoretical developments with a series of numerical examples.

\section{Notation and preliminaries}
\label{sec:notation}

Let us fix notation and the setting in which we will operate. In what follows $\Omega$ is a convex, open and bounded domain of $\R^n$ ($n \geq 1$) with polytopal boundary. The handling of curved boundaries is somewhat standard, but leads to additional technicalities that will only obscure the main ideas we are trying to advance. By $A \lesssim B$ we mean that there is a nonessential constant $c$ such that $A \leq c B$. The value of this constant might change at each occurrence.

\subsection{Weights and weighted spaces}
\label{sub:weights}
Throughout our discussion we call a \emph{weight} a function $\omega \in L^1_{loc}(\R^n)$ such that $\omega(x)>0$ for \mae $ x \in \R^n$. In particular we are interested in the so-called \emph{Muckenhoupt} weights \cite{MR1800316,Turesson}.

\begin{definition}[Muckenhoupt class]
\label{def:defofAp}
Let $r \in (1,\infty)$ and $\omega$ be a weight. We say that $\omega \in A_r$ if
\[
  C_{r,\omega} := \sup_B \left( \fint_B \omega(x) \diff x \right)\left( \fint_B \omega^{1/(1-r)}(x) \diff x \right)^{r-1} < \infty,
\]
where the supremum is taken over all balls $B \subset \R^n$.
\end{definition}

From the fact that $\omega \in A_r$ many fundamental consequences for analysis follow. For instance, the induced measure $\omega \diff x$ is not only doubling, but also \emph{strong doubling} (\cf \cite[Proposition 2.2]{NOS2}). We introduce
the \emph{weighted Lebesgue spaces}
\[
  L^r(\omega,\Omega) = \left\{ v \in L^0(\Omega) : \int_\Omega |v(x)|^r \omega(x) \diff x < \infty \right\}  ,
\]
and note that \cite[Proposition 2.3]{NOS2} shows that their elements are distributions, therefore we can
define \emph{weighted Sobolev spaces}
\[
  W^{k,r}(\omega,\Omega) = \left\{ v \in L^r(\omega,\Omega): D^\kappa v \in L^r(\omega,\Omega) \ \forall \kappa : |\kappa | \le k \right\},
\]
which are complete, separable and smooth functions are dense in them (\cf \cite[Proposition 2.1.2, Corollary 2.1.6]{Turesson}). We denote $H^1(\omega,\Omega) = W^{1,2}(\omega,\Omega)$.

We define $W^{k,r}_0(\omega,\Omega)$ as the closure of $C_0^\infty(\Omega)$ in $W^{k,r}(\omega,\Omega)$ and set $H^1_0(\omega,\Omega) = W^{1,2}_0(\omega,\Omega)$. On these spaces, the following \emph{Poincar\'e} inequality holds
\begin{equation}
\label{eq:Poincare}
  \| v \|_{L^r(\omega,\Omega)} \lesssim \| \GRAD v \|_{L^r(\omega,\Omega)} \quad \forall v \in W^{1,r}_0(\omega,\Omega),
\end{equation}
where the hidden constant is independent of $v$, depends on the diameter of $\Omega$ and depends on $\omega$ only through $C_{r,\omega}$.

The literature on the theory of Muckenhoupt weighted spaces is rather vast so we only refer the reader to \cite{MR1800316,NOS2,Turesson} for further results.

\subsection{Finite element approximation of weighted spaces}
\label{sub:FEM}
Since, the spaces $W^{1,r}(\omega,\Omega)$ are separable for $\omega \in A_r$ $(r > 1)$, and smooth functions are dense, it is possible to develop a complete approximation theory using functions that are piecewise polynomial. This is essential, for instance, to analyze the numerical approximation of \eqref{eq:defofPDEA} with finite element techniques. Let us then recall the main results from \cite{NOS2} concerning this scenario.

Let $\T = \{T\}$ be a conforming triangulation (into simplices or $n$-rectangles) of $\Omega$. We denote by $\Tr = \{\T\}$ a family of triangulations, which for simplicity we assume quasiuniform. The mesh size of $\T \in \Tr$ is denoted by $h_\T$. Given $\T \in \Tr$ we define the finite element space
\begin{equation}
\label{eq:V}
 \V(\T) = \left\{ v_\T \in C^0(\bar\Omega) : v_{\T|T} \in \calP(T), \ v_{\T|\partial\Omega} = 0  \right\}, 
\end{equation}
where, if $T$ is a simplex, $\calP(T) = \mathbb{P}_1(T)$ --- the space of polynomials of degree at most one. In the case that $T$ is an $n$-rectangle $\calP(T) = \mathbb{Q}_1(T)$ --- the space of polynomials of degree at most one in each variable. Notice that, by construction, $\V(\T) \subset W^{1,\infty}_0(\Omega) \subset W^{1,r}_0(\omega,\Omega)$ for any $r \in(1,\infty)$ and $\omega \in A_r$. 

The results of \cite{NOS2} show that there exists a quasi-interpolation operator $\Pi_\T: L^1(\Omega) \to \V(\T)$, which is based on local averages over stars and thus well defined for functions in $L^1(\Omega)$. This operator satisfies the following stability and approximation properties:
\begin{align*}
  \| \Pi_\T v \|_{L^r(\omega,\Omega)} &\lesssim \| v \|_{L^r(\omega,\Omega)}, &\forall v &\in L^r(\omega,\Omega), \\
  \| v - \Pi_\T v \|_{L^r(\omega,\Omega)} & \lesssim h_\T \| v \|_{W^{1,r}(\omega,\Omega)}, &\forall v &\in W^{1,r}(\omega,\Omega), \\
  \| \Pi_\T v \|_{W^{1,r}(\omega,\Omega)} &\lesssim \| v \|_{W^{1,r}(\omega,\Omega)}, &\forall v &\in W^{1,r}(\omega,\Omega), \\
  \| v - \Pi_\T v \|_{W^{1,r}(\omega,\Omega)} & \lesssim h_\T \| v \|_{W^{2,r}(\omega,\Omega)}, &\forall v &\in W^{2,r}(\omega,\Omega).
\end{align*}

Finally, to approximate the PDE constrained optimization problems described in Section \ref{sec:introduccion} we define the space of piecewise constants by
\begin{equation}
\label{eq:Upc}
\U(\T) = \left\{ v_\T \in L^\infty(\Omega): v_{\T|T} \in \mathbb{P}_0(T) \right\}.
\end{equation}

\subsection{Optimality conditions}
\label{sub:optim}
To unify the analysis and discretization of the PDE constrained optimization problems introduced and motivated in Section \ref{sec:introduccion} and thoroughly studied in subsequent sections, we introduce a general framework following the guidelines presented in \cite{GH:09,MR2516528,MR2441683,MR0271512,JCarlos,Tbook}. Let $\bbU$ and $\bbH$ be Hilbert spaces denoting the so-called control and observation spaces, respectively. We introduce the state trial and test spaces $\bbY_1$ and $\bbX_1$, and the corresponding adjoint trial and test spaces $\bbY_2$ and $\bbX_2$, which we assume to be Hilbert. In addition, we introduce:
\begin{enumerate}[(a)]
  \item \label{a} A bilinear form $a:(\bbY_1+\bbY_2) \times (\bbX_1 + \bbX_2) \to \R$ which, when restricted to either $\bbY_1\times \bbX_1$ or $\bbY_2 \times \bbX_2$, satisfies the conditions of the BNB theorem; see \cite[Theorem 2.6]{Guermond-Ern}.
  
  \item \label{b} A bilinear form $b: \bbU \times (\bbX_1 + \bbX_2) \to \R$ which, when restricted to either $\bbU \times \bbX_1$ or $\bbU \times \bbX_2$ is bounded. The bilinear forms $a$ and $b$ will be used to describe the state and adjoint equations.
  
  \item \label{c} An observation map $C: \dom(C) \subset \bbY_1 + \bbY_2 \to \bbH$, which we assume linear.
  
  \item \label{d} A desired state $y_d \in \bbH$.

  \item \label{e} A regularization parameter $\lambda > 0$ and a cost functional
  \begin{equation}
  \label{eq:absJ}
    \bbY_1 \times \bbU \ni (y,u) \mapsto J(y,u) = \frac12 \| Cy - y_d \|_{\bbH}^2 + \frac\lambda2 \| u \|_{\bbU}^2.
  \end{equation}
\end{enumerate}
All our problems of interest can be cast as follows. Find $\min J$ subject to: 
\begin{equation}
\label{eq:abstate}
y \in \bbY_1: \quad a(y,v) = b(u,v) \quad \forall v \in \bbX_1,
\end{equation}
and the constraints
\begin{equation}
\label{eq:ccabs}
u \in \Uad,
\end{equation}
where $\Uad \subset \bbU$ is nonempty, bounded, closed and convex. We introduce the control to state map $S : \bbU \rightarrow \bbY_1$ which to a given control, $u \in \bbU$, associates a unique state, $y(u) = Su \in \bbY_1$, that solves the state equation \eqref{eq:abstate}. As a consequence of \eqref{a} and \eqref{b}, the map $S$ is a bounded and linear operator. With this operator at hand we can eliminate the state variable $y$ from \eqref{eq:absJ} and introduce the reduced cost functional
\begin{equation}
\label{eq:redcostabs}
  \bbU \ni u \mapsto j(u) = \frac12 \| CSu - y_d \|_{\bbH}^2 + \frac\lambda2 \| u \|_{\bbU}^2.
\end{equation}
Then, our problem can be cast as: Find $\min j$ over $\Uad$. As described in \eqref{e} we have $\lambda > 0$ so that $j$ is strictly convex. In addition, $\Uad$ is weakly sequentially compact in $\bbU$. Consequently, standard arguments yield existence and uniqueness of a minimizer \cite[Theorem 2.14]{Tbook}. In addition, the optimal control $\usf \in \Uad$ can be characterized by the variational inequality
\[
  j'(\usf)(u - \usf) \ge 0 \quad \forall u \in \Uad,
\]
where $j'(w)$ denotes the G\^ateaux derivative of $j$ at $w$ \cite[Lemma 2.21]{Tbook}. This variational inequality can be equivalently written as
\begin{equation}
\label{eq:VIabs}
  b( u - \usf, \psf) + \lambda (\usf,u - \usf)_\bbU \geq 0 \quad \forall u \in \Uad,
\end{equation}
where $\psf = \psf(\usf)$ denotes the optimal adjoint state and solves
\begin{equation}
\label{eq:absadj}
  \psf \in \bbX_2: \quad a(v,\psf) = (C \ysf - y_d, Cv)_\bbH \quad \forall v \in \bbY_2.
\end{equation}
The optimal state $\ysf = \ysf(\usf) \in \bbY_1$ is the solution to \eqref{eq:abstate} with $u = \usf$. 

\subsection{Discretization of PDE constrained optimization problems}
\label{sub:absdisc}
Let us now, in the abstract setting of Section \ref{sub:optim}, study the discretization of problem \eqref{eq:absJ}--\eqref{eq:ccabs}. Since our ultimate objective is to approximate the problems described in Section \ref{sec:introduccion} with finite element methods, we will study the discretization of \eqref{eq:absJ}--\eqref{eq:ccabs} with Galerkin-like techniques.

Let $h>0$ be a parameter and assume that, for every $h>0$, we have at hand finite dimensional spaces $\bbU^h \subset \bbU$, $\bbX_1^h \subset \bbX_1$, $\bbX_2^h \subset \bbX_2$, $\bbY_1^h \subset \bbY_1$ and $\bbY_2^h \subset \bbY_2$. We define $\Uad^h = \bbU^h \cap \Uad$, which we assume nonempty. About the pairs $(\bbX_i^h,\bbY_i^h)$, for $i=1,2$, we assume that they are such that $a$ satisfies a BNB condition uniformly in $h$; see \cite[\S 2.2.3]{Guermond-Ern}. 
In this setting, the discrete counterpart of \eqref{eq:absJ}--\eqref{eq:ccabs} reads: Find
\begin{equation}
\label{eq:Jdis}
\min J(y_h,u_h)
\end{equation}
subject to the discrete state equation
\begin{equation}
\label{eq:abstateh}
y_h \in \bbY_1^h: \quad a(y_h, v_h ) = b(u_h, v_h) \quad \forall v_h \in \bbX_1^h,
\end{equation}
and the discrete constraints
\begin{equation}
\label{eq:ccdis} 
u_h \in \Uad^h.
\end{equation}

As in the continuous case, we introduce the discrete control to state operator $S_h$, which to a discrete control, $u_h \in \bbU_h$, associates a unique discrete state, $y_h = y_h (u_h) = S_h u_h$, that solves \eqref{eq:abstateh}. $S_h$ is a bounded and linear operator.

The pair $(\ysf_h,\usf_h) \in \bbY_1^h \times \Uad^h$ is optimal for \eqref{eq:Jdis}--\eqref{eq:ccdis} if $\ysf_h = \ysf_h(\usf_h)$ solves \eqref{eq:abstateh} and the discrete control $\usf_h$ satisfies the variational inequality
\[
  j_h'(\usf_h)(u_h - \usf_h) \ge 0 \quad \forall u_h \in \Uad^h,
\]
or, equivalently, 
\begin{equation}
\label{eq:VIabsh}
  b(u_h - \usf_h, \psf_h) + \lambda (\usf_h, u_h - \usf_h)_\bbU \geq 0 \quad \forall u_h \in \Uad^h,
\end{equation}
where the discrete adjoint variable $\psf_h = \psf_h(\usf_h)$ solves
\begin{equation}
\label{eq:absadjh}
\psf_h \in \bbX_2^h: \quad  a(v_h,\psf_h) = (C \ysf_h - y_d, C v_h )_\bbH \quad \forall v_h \in \bbY_2^h.
\end{equation}

To develop an error analysis for the discrete problem described above, we introduce $\Pi_{\bbU}$ the $\bbU$-orthogonal projection onto $\bbU^h$. We assume that $\Pi_\bbU \Uad \subset \Uad^h$. In addition, we introduce two auxiliary states that will play an important role in the discussion that follows. We define 
\begin{equation}
\label{eq:hatyh}
\hat{\wbysf}_h \in \bbY_1^h: \quad a(\hat{\wbysf}_h , v_h ) = b(\usf, v_h) \quad \forall v_h \in \bbX_1^h,
\end{equation}
\ie $\hat{\wbysf}_h$ is defined as the solution to \eqref{eq:abstateh} with $u_h$ replaced by $\usf$. We also define
\begin{equation}
\label{eq:hatph}
\hat{\wbpsf}_h \in \bbX_2^h: \quad  a(v_h,\psf_h) = (C \hat{\wbysf}_h - y_d, C v_h )_\bbH \quad \forall v_h \in \bbY_2^h,
\end{equation}
this is, $\hat{\wbpsf}_h$ is the solution to \eqref{eq:absadjh} with $\ysf_h$ replaced by $\hat{\wbysf}_h$.

The main error estimate with this level of abstraction reads as follows.

\begin{lemma}[abstract error estimate]
\label{lem:abserror}
Let $(\ysf,\usf) \in \bbY_1 \times \U_{\mathrm{ad}}$ and $(\ysf_h,\usf_h) \in \bbY_1^h \times \bbU_{\mathrm{ad}}^h$ be the continuous and discrete optimal pairs that solve \eqref{eq:absJ}--\eqref{eq:ccabs} and \eqref{eq:Jdis}--\eqref{eq:ccdis}, respectively. 
If
\begin{equation}
\label{eq:intersection}
  \psf_h - \hat{\wbpsf}_h \in \bbX_1^h \cap \bbX_2^h, \qquad \ysf_h - \hat{\wbysf}_h \in \bbY_1^h \cap \bbY_2^h,
\end{equation}
then, for any $\varepsilon>0$, we have the estimate
\begin{equation}
\label{eq:errest}
  \begin{aligned}
    \| \usf - \usf_h \|_{\bbU}^2 &\leq c_1 \left( \| \psf - \hat{\wbpsf}_h  \|_{\bbX_2}^2 + j'(\usf)(\Pi_\bbU \usf - \usf) 
        + \| \Pi_\bbU \usf - \usf  \|_{\bbU}^2 \right) \\
        &+ \varepsilon \left( \sup_{v_p \in \bbY_2^h} \frac{ (C ( \ysf_h - \hat{\wbysf}_h), C v_p )_\bbH}{\| v_p \|_{\bbY_2}} \right)^2,
  \end{aligned}
\end{equation}
where the constant $c_1$ depends on $\lambda^{-1}$ and $\varepsilon^{-1}$ but does not depend on $h$.
\end{lemma}
\begin{proof}
Since by definition $\Uad^h \subset \Uad$ and by assumption $\Pi_{\bbU} \Uad \subset \Uad^h$, we set $u=\usf_h$ and $u_h = \Pi_\bbU \usf$ in \eqref{eq:VIabs} and \eqref{eq:VIabsh} respectively. Adding the ensuing inequalities we obtain
\begin{equation}\
\label{eq:aux}
 \lambda \| \usf - \usf_h \|_{\bbU}^2 \leq b(\usf_h - \usf, \psf - \psf_h) 
  +b( \Pi_\bbU \usf - \usf, \psf_h) +  \lambda (\usf_h, \Pi_\bbU \usf - \usf )_{\bbU}. 
\end{equation}

Denote $\mathrm{I} = b(\usf_h - \usf, \psf - \psf_h)$. In order to estimate this term we add and subtract $\hat{\wbpsf}_h$ to obtain
\begin{equation}
\label{I}
\textrm{I}  = b( \usf_h - \usf, \psf - \hat{\wbpsf}_h ) + b( \usf_h - \usf, \hat{\wbpsf}_h - \psf_h ).
\end{equation}
Since $\hat{\wbpsf}_h$ is the unique solution to \eqref{eq:hatph}, 
we have that
\begin{equation}
\label{eq:phMphu}
  a( v_p, \psf_h - \hat{\wbpsf}_h ) = ( C( \ysf_h - \hat{\wbysf}_h ), C v_p )_{\bbH} \quad \forall v_p \in \bbY_2^h.
\end{equation}
Similarly, since $\ysf_h$ solves \eqref{eq:hatyh}, 
we derive
\[
  a( \ysf_h - \hat{\wbysf}_h , v_y ) = b( \usf_h - \usf, v_y ) \quad \forall v_y \in \bbX_1^h.
\]
Set $v_p = \ysf_h - \hat{\wbysf}_h$ and $v_y = \psf_h - \hat{\wbpsf}_h$ which, by assumption \eqref{eq:intersection}, are admissible test functions to obtain
\[
  b( \usf_h - \usf, \hat{\wbpsf}_h - \psf_h  ) = ( C( \ysf_h - \hat{\wbysf}_h), C( \hat{\wbysf}_h - \ysf_h) )_{\bbH} \leq 0.
\]
This, and the continuity of the bilinear form $b$ allow us to bound \eqref{I} as follows:
\begin{equation*}
  \textrm{I} \leq b( \usf_h - \usf, \psf - \hat{\wbpsf}_h ) \leq \frac\lambda4 \| \usf - \usf_h \|_{\bbU}^2 
  +  \frac{\| b \|^2}{\lambda} \| \psf - \hat{\wbpsf}_h \|_{\bbX_2}^2, 
\end{equation*}
where $\|b\|$ denotes the norm of the bilinear form $b$.

Let us now analyze the remaining terms in \eqref{eq:aux}, which we denote by $\textrm{II}$. To do this, we rewrite $\textrm{II}$ as follows:
\begin{multline*}
  \textrm{II} = b(\Pi_{\bbU} \usf - \usf, \psf) + \lambda ( \usf, \Pi_{\bbU} \usf - \usf)_{\bbU}
     + \lambda (\usf_h - \usf, \Pi_{\bbU} \usf - \usf )_{\bbU} \\
     + b(\Pi_{\bbU} \usf - \usf, \hat{\wbpsf}_h - \psf )
     + b(\Pi_{\bbU} \usf - \usf, \psf_h - \hat{\wbpsf}_h ).
\end{multline*}
Now, notice that
\[
  b(\Pi_{\bbU} \usf - \usf, \psf) + \lambda ( \usf, \Pi_{\bbU} \usf - \usf)_{\bbU} = j'(\usf)(\Pi_{\bbU} \usf - \usf)
\]
and
\[
  \lambda (\usf_h - \usf, \Pi_{\bbU} \usf - \usf )_{\bbU} \leq \frac\lambda4 \| \usf - \usf_h \|_{\bbU}^2 + \frac1\lambda \| \usf - \Pi_{\bbU} \usf \|_{\bbU}^2.
\]
Next, since the bilinear form $b$ is continuous, we arrive at
\[
  b(\Pi_{\bbU} \usf - \usf, \hat{\wbpsf}_h  - \psf ) \leq \frac{\| b\|}{2} \| \Pi_{\bbU} \usf - \usf \|_{\bbU}^2 
  + \frac{\| b\|}{2} \| \psf - \hat{\wbpsf}_h \|_{\bbX_2}^2.
\]
The remaining term, which we will denote by $\textrm{III}$, is treated by using, again, that the bilinear form $b$ is continuous. This implies that for any $\varepsilon > 0$
\[
  \textrm{III}:= b(\Pi_{\bbU} \usf - \usf, \psf_h - \hat{\wbpsf}_h) \leq \frac{\|b\|}{2\varepsilon} \| \Pi_{\bbU} \usf - \usf \|_{\bbU}^2 
  + \|b\|\frac{\varepsilon}2 \| \psf_h - \hat{\wbpsf}_h \|_{\bbX_2}^2 . 
\]
From \eqref{eq:phMphu} and the fact that the discrete spaces satisfy a discrete BNB condition uniformly in $h$ we conclude
\[
  \| \psf_h - \hat{\wbpsf}_h \|_{\bbX_2} \lesssim \sup_{v_p \in \bbY_2^h} \frac{ (C (\ysf_h - \hat{\wbysf}_h), C v_p )_{\bbH}}{\| v_p \|_{\bbY_2^h}}.
\]
Collecting these derived estimates we bound the term $\textrm{II}$.

The estimates we obtained for $\textrm{I}$, $\textrm{II}$ and $\textrm{III}$ readily yield the claimed result.
\end{proof}

The use of this simple result will be illustrated in the following sections.

\begin{remark}[on the choice of $\varepsilon$]
\rm
In \eqref{eq:errest} the choice of $\varepsilon$ can be arbitrary. A judicious choice will be fundamental in the error analysis of the optimization problem with point observations \eqref{eq:defofJp}--\eqref{eq:cc2}.
\end{remark}

\section{Optimization with nonuniformly elliptic equations}
\label{sec:A}
In this section we study the problem \eqref{eq:defofJA}--\eqref{eq:cc} under the abstract framework developed in Section \ref{sub:optim}. Let $\Omega \subset \R^n$ be a convex polytope $(n \geq 1)$ and $\omega \in A_2(\mathbb{R}^n)$ where the $A_2$-Muchenkhoupt class is given by Definition~\ref{def:defofAp}. In addition, we assume that $\A: \Omega \to \mathbb{M}^n$ is symmetric and satisfies the nonuniform ellipticity condition \eqref{eq:nonunifellip}.

\subsection{Analysis}
\label{sub:analysisA}
Owing to the fact that the diffusion matrix $\A$ satisfies \eqref{eq:nonunifellip} with $\omega \in A_2(\mathbb{R}^n)$, as shown in \cite{FKS:82}, the state equation \eqref{eq:defofPDEA} is well posed in $H^1_0(\omega,\Omega)$, whenever $u \in L^2(\omega^{-1},\Omega)$. For this reason, we set:
\begin{enumerate}[$\bullet$]
  \item $\bbH = L^2(\omega,\Omega)$ and $C = \id$.

  \item $\bbU = L^2(\omega^{-1},\Omega)$.
  
  \item $\bbX_1 = \bbX_2 = \bbY_1 = \bbY_2 = H^1_0(\omega,\Omega)$, and
  \[
    a(v_1,v_2) = \int_\Omega \GRAD v_2(x)^\intercal \A(x) \GRAD v_1(x) \diff x,
  \]
  which, as a consequence of \eqref{eq:nonunifellip} with $\omega \in A_2(\mathbb{R}^n)$ and the Poincar\'e inequality \eqref{eq:Poincare}, is bounded, symmetric and coercive in $H^1_0(\omega,\Omega)$.
  
  \item $b(\cdot,\cdot) = (\cdot,\cdot)_{L^2(\Omega)}$. Notice that, if $v_1 \in L^2(\omega^{-1},\Omega)$ and $v_2 \in H^1_0(\omega,\Omega)$ then
  \[
    b(v_1,v_2) = (v_1, v_2)_{L^2(\Omega)} \leq \| v_1 \|_{L^2(\omega^{-1},\Omega)} \| v_2 \|_{L^2(\omega,\Omega)} \lesssim
    \| v_1 \|_{L^2(\omega^{-1},\Omega)} \| \nabla v_2 \|_{L^2(\omega,\Omega)},
  \]
  where we have used the Poincar\'e inequality \eqref{eq:Poincare}.

    \item The cost functional as in \eqref{eq:defofJA}.
\end{enumerate}

For $\asf,\bsf \in \R$, $\asf < \bsf$ we define the set of admissible controls by
\begin{equation}
\label{eq:UA}
 \mathbb{U}_{\A} = \left\{ u \in L^2(\omega^{-1},\Omega): \asf \leq u \leq \bsf \ \mae x \in \Omega \right\}, 
\end{equation}
which is closed, bounded and convex in $L^2(\omega^{-1},\Omega)$. In addition, since $\lambda > 0$ the functional
\eqref{eq:defofJA} is strictly convex. Consequently, the optimization problem with nonuniformly elliptic state equation \eqref{eq:defofJA}--\eqref{eq:cc} has a unique optimal pair $(\ysf,\usf) \in H^1_0(\omega,\Omega) \times L^2(\omega^{-1},\Omega)$ \cite[Theorem 2.14]{Tbook}. In this setting, 
the first order necessary and sufficient optimality condition \eqref{eq:VIabs} reads
\begin{equation}
\label{eq:VIA}
  (\psf, u - \usf)_{L^2(\Omega)} + \lambda (\usf, u - \usf)_{L^2(\omega^{-1},\Omega)} \geq 0 \quad \forall u \in \bbU_\A,
\end{equation}
where the optimal state $\ysf = \ysf(\usf) \in H_0^1(\omega,\Omega)$ solves
\begin{equation}
\label{eq:stateA}
  a(\ysf,v) = (\usf,v)_{L^2(\Omega)} \quad \forall v \in H^1_0(\omega, \Omega)
\end{equation}
and the optimal adjoint state $\psf = \psf(\usf) \in H^1_0(\omega,\Omega)$ solves
\begin{equation}
\label{eq:adjA}
  a(v, \psf ) = (\ysf - y_d,v)_{L^2(\omega,\Omega)} \quad \forall v \in H^1_0(\omega,\Omega).
\end{equation}
The results of \cite{FKS:82}, again, yield that the adjoint problem is well posed.

\subsection{Discretization}
\label{sub:discrA}

Let us now propose a discretization for problem \eqref{eq:defofJA}--\eqref{eq:cc}, and derive a priori error estimates based on the results of Section \ref{sub:absdisc}. Given a family $\Tr = \{\T\}$ of quasi-uniform triangulations of $\Omega$ we set:
\begin{enumerate}[$\bullet$]
  \item $\bbU^h = \U(\T)$, where the discrete space $\U(\T)$ is defined in \eqref{eq:Upc}.
  
  \item $\Uad^h = \U^h \cap \U_\A$, where the set of admissible controls $\U_\A$ is defined in \eqref{eq:UA}.
  
  \item $\Pi_\bbU$ is the $L^2(\omega^{-1},\Omega)$-orthogonal projection onto $\U(\T)$, which we denote by $ \Pi_{\omega^{-1}}$ and is defined by
  \begin{equation}
  \label{eq:proj-1}
 (\Pi_{\omega^{-1}} v)_{|T} = \frac1{\omega^{-1}(T)} \int _T \omega^{-1}(x) v(x) \diff x \quad \forall T \in \T .    
  \end{equation}

  The definition of $\bbU_{\A}$ yields that $\Pi_{\omega^{-1}} \U_\A \subset \Uad^h$.
  
  \item $\bbX_1^h = \bbX_2^h = \bbY_1^h = \bbY_2^h = \V(\T)$, where the discrete space $\V(\T)$ is defined in \eqref{eq:V}.
\end{enumerate}

Notice that, since $\bbX_1^h = \bbX_2^h = \bbY_1^h = \bbY_2^h$, assumption \eqref{eq:intersection} is trivially satisfied. We obtain the following a priori error estimate.

\begin{corollary}[a priori error estimate]
\label{col:errA}
Let $\usf$ and $\usf_h$ be the continuous and discrete optimal controls, respectively. If $\ysf, \psf \in H^2(\omega,\Omega)$ then
\begin{align*}
\| \usf - \usf_h \|_{L^2(\omega^{-1},\Omega)} & \lesssim \left\| \usf - \Pi_{\omega^{-1}} \usf \right\|_{L^2(\omega^{-1},\Omega)} + \| \omega \psf - \Pi_{\omega^{-1}}(\omega\psf) \|_{L^2(\omega^{-1},\Omega)}\\
& + h_\T (  \| \ysf \|_{H^2(\omega,\Omega)} + \| \psf \|_{H^2(\omega,\Omega)}),
\end{align*}
where the hidden constant is independent of $h_\T$.
\end{corollary}
\begin{proof}
We invoke Lemma \ref{lem:abserror} and bound each one of the terms in \eqref{eq:errest}. First, since $\ysf, \psf \in H^2(\omega,\Omega)$, the results of \cite{NOS2} imply that
\[
  \| \psf - \hat{\wbpsf}_h \|_{H^1(\omega,\Omega)} \lesssim h_\T \left(  \| \ysf \|_{H^2(\omega,\Omega)} + \| \psf \|_{H^2(\omega,\Omega)} \right).
\]
Indeed, since $\psf$ solves \eqref{eq:adjA} and $\hat{\wbpsf}_h$ solves \eqref{eq:hatph},
the term $\psf - \hat{\wbpsf}_h$ satisfies
\[
  a(v_h,\psf - \hat{\wbpsf}_h ) = ( \ysf - \hat{\wbysf}_h, v_h )_{L^2(\omega,\Omega)} \quad \forall v_h \in \V(\T).
\]
Adding and subtracting the terms $\Pi_\T \psf$ and $\psf$ appropriately, where $\Pi_\T$ denotes the interpolation operator described in \S\ref{sub:FEM}, and using the coercivity of $a$ we arrive at
\[
  \| \psf - \hat{\wbpsf}_h \|_{H^1_0(\omega,\Omega)} \lesssim \| \psf - \Pi_\T \psf \|_{H^1_0(\omega,\Omega)} + \| \ysf - \hat{\wbysf}_h \|_{ H^1_0(\omega,\Omega)}.
\]
Using the regularity of $\psf$ and $\ysf$ we obtain the claimed bound.

We now handle the second term involving the derivative of the reduced cost $j$. Since it can be equivalently written using \eqref{eq:VIabs}, invoking the definition of $\Pi_{\omega^{-1}}$ given by \eqref{eq:proj-1}, we obtain
\begin{align*}
  j'(\usf)(\Pi_{\omega^{-1}} \usf -\usf ) &=
  (\psf,\Pi_{\omega^{-1}}\usf - \usf)_{L^2(\Omega)} + \lambda (\usf,\Pi_{\omega^{-1}} \usf -\usf)_{L^2(\omega^{-1},\Omega)} 
  \\ &=
  (\omega\psf - \Pi_{\omega^{-1}}(\omega\psf), \Pi_{\omega^{-1}}\usf - \usf)_{L^2(\omega^{-1},\Omega)}
  - \lambda \| \Pi_{\omega^{-1}} \usf -\usf \|_{L^2(\omega^{-1},\Omega)}^2
  \\ &\lesssim
  \| \omega \psf - \Pi_{\omega^{-1}}(\omega\psf) \|_{L^2(\omega^{-1},\Omega)}^2 + \| \Pi_{\omega^{-1}} \usf -\usf \|_{L^2(\omega^{-1},\Omega)}^2 . 
\end{align*}

The Poincar\'e inequality \eqref{eq:Poincare}, in conjunction with the stability of the discrete state equation \eqref{eq:abstateh}, yield
\begin{align*}
  ( \ysf_h - \hat{\wbysf}_h , v_h)_{L^2(\omega,\Omega)} &\lesssim \| \ysf_h - \hat{\wbysf}_h \|_{H^1_0(\omega,\Omega)} \| v_h \|_{H^1_0(\omega,\Omega)} \\
  &\lesssim \|\usf - \usf_h \|_{L^2(\omega^{-1},\Omega)} \| v_h \|_{H^1_0(\omega,\Omega)},
\end{align*}
for all $v_h \in \V(\T)$. This yields control of the last term in \eqref{eq:errest}.

These bounds yield the result.
\end{proof}

\begin{remark}[regularity of $\ysf$ and $\psf$] \rm
\label{rem:regpA}
The results of Corollary~\ref{col:errA} rely on the fact that $\ysf, \psf \in H^2(\omega,\Omega)$. Reference \cite{MR2780884} provides sufficient conditions for this to hold.
\end{remark}

\begin{theorem}[rate of convergence]
In the setting of Corollary~\ref{col:errA}, if we additionally assume that $\omega \psf \in H^1(\omega^{-1},\Omega)$ then, we have the optimal error estimate
\[
 \| \usf - \usf_h \|_{L^2(\omega^{-1},\Omega)}
  \lesssim h_\T \left( \| \ysf \|_{H^2(\omega,\Omega)} + \| \psf \|_{H^2(\omega,\Omega)} +\| \omega \psf \|_{H^1(\omega^{-1},\Omega)} + \| \usf  \|_{H^1(\omega^{-1},\Omega)} \right),
\]
where the hidden constant is independent of $h_\T$. 
\end{theorem}
\begin{proof}
We bound $\left\| \usf - \Pi_{\omega^{-1}} \usf \right\|_{L^2(\omega^{-1},\Omega)}$ and $\| \omega \psf - \Pi_{\omega^{-1}}(\omega\psf) \|_{L^2(\omega^{-1},\Omega)}$. 
Using that $\omega\psf \in H^1(\omega^{-1},\Omega)$ and a Poincar\'e-type inequality \cite[Theorem 6.2]{NOS2}, we obtain
\[
 \| \omega \psf - \Pi_{\omega^{-1}}(\omega\psf) \|_{L^2(\omega^{-1},\Omega)} \lesssim h_{\T} \| \omega \psf \|_{H^1(\omega^{-1},\Omega)}.
\]
Now, to estimate the term $\Pi_{\omega^{-1}} \usf -\usf$, it is essential to understand the regularity properties of $\usf$. From \cite[Section 3.6.3]{Tbook}, $\usf$ solves \eqref{eq:VIA} if and only if
\[
  \usf = \max \left\{ \asf, \min\left\{ \bsf, -\frac1\lambda \omega\psf\right\} \right\}.
\]
The assumption $\omega\psf \in H^1(\omega^{-1},\Omega)$ immediately yields $\usf \in H^1(\omega^{-1},\Omega)$ \cite[Theorem A.1]{KSbook}, which allows us to derive the estimate 
\[
 \| \usf - \Pi_{\omega^{-1}} \usf \|_{L^2(\omega^{-1},\Omega)} \lesssim h_{\T}\| \usf  \|_{H^1(\omega^{-1},\Omega)}.
\]
Collecting the derived results we arrive at the desired estimate.
\end{proof}

\section{Optimization with point observations}
\label{sec:points}
Here we consider problem \eqref{eq:defofJp}--\eqref{eq:cc2}. Let $\Omega \subset \R^n$ be a convex polytope, with $n \in \{ 2,3\}$. We recall that $\mathcal{Z} \subset \Omega$ denotes the set of \emph{observable points} with $\# \mathcal{Z} < \infty$. 

\subsection{Analysis}
\label{sub:analysisp}
To analyze problem \eqref{eq:defofJp}--\eqref{eq:cc2} using the framework of weighted spaces we must begin by defining a suitable weight. Since $\#\calZ < \infty$, we know that $d_{\mathcal{Z}} = \min\{|z-z'|: z,z' \in \calZ, \ z\neq z' \} > 0$. For each $z \in \calZ$ we then define
\[
  \dist_z(x) = \frac1{2d_{\mathcal{Z}}} | x - z |, \qquad \varpi_z(x) = \frac{\dist_z(x)^{n-2}}{\log^2\dist_z(x)}
\]
and the weight
\begin{equation}
\label{eq:defofoweight}
  \varpi(x) = \begin{dcases}
                \varpi_z(x), & \exists z \in \calZ: \dist_z(x) < \frac12, \\
                \frac{2^{2-n}}{\log^2 2}, & \text{otherwise}.
              \end{dcases}
\end{equation}
As \cite[Lemma 7.5]{NOS2} shows, with this definition we have that $\varpi \in A_2$. With this $A_2$-weight at hand we set:
\begin{enumerate}[$\bullet$]
  \item $\bbH = \R^{\# \mathcal{Z}}$ and $C = \sum_{z \in \calZ} \mathbf{e}_z \delta_z $, where $\{\mathbf{e}_z\}_{z \in \calZ}$ is the canonical basis of $\bbH$.
  
  \item $\bbU = L^2(\Omega)$.
  
  \item $\bbX_1 = \bbY_1 = H^1_0(\Omega)$.
  
  \item $\bbX_2 = H^1_0(\varpi,\Omega)$ and $\bbY_2 = H^1_0(\varpi^{-1},\Omega)$ and
  \[
    a(v,w) = \int_\Omega \GRAD v(x)^\intercal \cdot \GRAD w(x) \diff x,
  \]
  which is bounded, symmetric and coercive in $H_0^1(\Omega)$ and satisfies the conditions of the BNB theorem in $H_0^1(\varpi,\Omega) \times H^1_0(\varpi^{-1},\Omega)$ \cite[Lemma 7.7]{NOS2}.
  
  \item $b(\cdot,\cdot) = (\cdot,\cdot)_{L^2(\Omega)}$. The results of \cite[Lemma 7.6]{NOS2} guarantee that, for $n<4$, the embedding $H^1_0(\varpi,\Omega) \hookrightarrow L^2(\Omega)$ holds. Therefore,
  \[
    b(v_1,v_2) \lesssim \| v_1 \|_{L^2(\Omega)} \| v_2 \|_{H^1_0(\varpi,\Omega)}.
  \]
\end{enumerate}

For $\asf, \bsf \in \R$ with $\asf < \bsf$ we define the set of admissible controls by
\begin{equation}
\label{eq:Uz}
\bbU_\calZ = \left\{ u \in L^2(\Omega): \asf \leq u \le \bsf,\ \mae x \in \Omega \right\}.
\end{equation}

With this notation, the pair $(\ysf,\usf) \in H_0^1(\Omega) \times L^2(\Omega)$ is optimal for problem
\eqref{eq:defofJp}--\eqref{eq:cc2} if and only if $\ysf$ solves
\begin{equation}
\label{eq:statep}
 \ysf \in H^1_0(\Omega): \quad   a(\ysf,w) = (\usf, w )_{L^2(\Omega)} \quad \forall w \in H^1_0(\Omega),
\end{equation}
and the optimal control $\usf$ satisfies
\begin{equation}
\label{eq:VIp}
  ( \psf, u - \usf )_{L^2(\Omega)} + \lambda (\usf, u - \usf)_{L^2(\Omega)} \geq 0 \quad \forall u \in \bbU_\calZ,
\end{equation}
where the adjoint variable $\psf\in H^1_0(\varpi,\Omega)$ satisfies, for every $w \in H^1_0(\varpi^{-1},\Omega)$, 
\begin{equation}
\label{eq:adjp}
  a(w,\psf) = \sum_{z \in \calZ} ( \ysf(z) - y_z) \langle \delta_z , w\rangle_{H^1_0(\varpi^{-1},\Omega)' \times H^1_0(\varpi^{-1},\Omega)}.
\end{equation}
Notice that, since $\Omega$ is a convex polytope and $n<4$, we have that $\ysf \in H^2(\Omega)\hookrightarrow C(\bar\Omega)$, so point evaluations are meaningful. Therefore, since $\delta_z \in H^1_0(\varpi,\Omega)'$ for $z \in \Omega$, using \cite[Lemma 7.7]{NOS2} we have that the adjoint problem is well posed.

\subsection{Discretization}
\label{sub:discrp}
For a family $\Tr = \{\T\}$ of quasi-uniform meshes of $\Omega$ we set:
\begin{enumerate}[$\bullet$]
  \item $\bbU^h = \U(\T)$, where $\U(\T)$ is defined in \eqref{eq:Upc} and $\Uad^h = \U(\T) \cap \bbU_\calZ$, where $ \bbU_\calZ$ is defined in \eqref{eq:Uz}. The operator $\Pi_{\bbU} = \Pi_{L^2}$ is the standard $L^2(\Omega)$-projection:
  \[
    (\Pi_{L^2} v )_{|T} = \fint_T v(x) \diff x \quad \forall T \in \T.
  \]

  \item $\bbX_1^h = \bbX_2^h = \bbY_1^h = \bbY_2^h = \V(\T)$.
\end{enumerate}

To shorten the exposition we define
\begin{equation}
\label{eq:defofsigma}
  \sigma_\T = h_\T^{2-n/2}|\log h_\T| .
\end{equation}

With this notation, the error estimate for the approximation \eqref{eq:Jdis}--\eqref{eq:ccdis} to problem \eqref{eq:defofJp}--\eqref{eq:cc2} reads as follows.

\begin{corollary}[a priori error estimates]
\label{col:errp}
Let $\usf$ and $\usf_h$ be the continuous and discrete optimal controls, respectively. Assume that $h_{\T}$ is sufficiently small. If $n=2$, then we have
\begin{equation}
\label{eq:tracking_estimate1}
\| \usf - \usf_h \|_{L^2(\Omega)}  \lesssim \| \usf - \Pi_{L^2} \usf \|_{L^2(\Omega)} + \sigma_\T \left(\| \GRAD\psf \|_{L^2(\varpi,\Omega)} +  \| \GRAD \ysf \|_{L^\infty(\Omega)} \right).
\end{equation}
If $n=3$, we have
\begin{align} \label{eq:tracking_estimate2} 
\| \usf - \usf_h \|_{L^2(\Omega)}  \lesssim &  |\log h_\T| \| \usf - \Pi_{L^2} \usf \|_{L^2(\Omega)}
 + \sigma_\T \left(\| \GRAD\psf \|_{L^2(\varpi,\Omega)} +  \| \GRAD \ysf \|_{L^\infty(\Omega)} \right),
\end{align}
where $\sigma_\T$ is defined in \eqref{eq:defofsigma} and the hidden constants are independent of $\T$.
\end{corollary}
\begin{proof}
We follow Lemma~\ref{lem:abserror} with slight modifications. The term $\textrm{I}$
in \eqref{I} is estimated in two steps. In fact, since $(\usf_h - \usf, \hat{\wbpsf}_h - \psf_h)_{L^2(\Omega)} \leq 0$, we have
\[
 \textrm{I} \leq (\usf_h - \usf, \psf - \hat{\wbpsf}_h )_{L^2(\Omega)} \leq \frac\lambda4 \| \usf_{h} - \usf \|_{L^2(\Omega)}^2
  + \frac1\lambda \| \psf - \hat{\wbpsf}_h \|_{L^2(\Omega)}^2.
\]
We now analyze the second term of the previous expression. Let $\qsf_h \in \V(\T)$ solve
\begin{equation}
\label{eq:qh}
  a(w_h,\qsf_h ) = \sum_{z \in \calZ} (\ysf(z) - y_z) w_h(z) \quad \forall w_h \in \V(\T).
\end{equation}
The conclusion of \cite[Corollary 7.9]{NOS2} immediately yields
\[
  \| \psf - \qsf_h \|_{L^2(\Omega)} \lesssim \sigma_\T \| \GRAD\psf \|_{L^2(\varpi,\Omega)},
\]
so that it remains to estimate $\qsf_h - \hat{\wbpsf}_h$. We now invoke \cite[Theorem 6.1]{NOS2} with $p = q = 2$, $\rho = 1$ and $\omega = \varpi$. Under this setting the compatibility condition \cite[inequality (6.2)]{NOS2} is satisfied, and then \cite[Theorem 6.1]{NOS2} yields
\[
\| \qsf_h - \hat{\wbpsf}_h \|_{L^2(\Omega)} \lesssim \| \GRAD( \qsf_h - \hat{\wbpsf}_h ) \|_{L^2(\varpi,\Omega),}
\]
where the hidden constant depends on $\Omega$ and the quotient between the radii of the balls inscribed and circumscribed in $\Omega$. Since $\qsf_h$ solves \eqref{eq:qh}, the discrete inf-sup conditions of \cite[Lemma 7.8]{NOS2} and the fact that $\delta_z \in H^1_0(\varpi^{-1},\Omega)'$ yield
\[
  \| \GRAD( \qsf_h - \hat{\wbpsf}_h ) \|_{L^2(\varpi,\Omega)}
  \lesssim \| \ysf - \ysf_h(\usf) \|_{L^\infty(\Omega)}.
\]
We now recall that $\hat{\wbysf}_h$ is the Galerkin projection of $\ysf$. In addition, since $n \in \{2,3\}$, $\Omega$ is a convex polytope and $\usf \in L^\infty(\Omega)$, we have that $\ysf \in W^{1,\infty}(\Omega)$ (\cf \cite{MR1156467,MR2495783,MR1143406}). Therefore standard pointwise estimates for finite elements \cite[Theorem 5.1]{MR637283} yield
\begin{equation}
  \| \ysf - \hat{\wbysf}_h \|_{L^\infty(\Omega)} \lesssim h_\T |\log h_\T| \| \GRAD \ysf \|_{L^\infty(\Omega)}.
  \label{eq:y_Linf1}
\end{equation}
In conclusion,
\[
\textrm{I} \leq \frac\lambda4 \| \usf_h - \usf \|_{L^2(\Omega)}^2 + c \sigma_\T^2 \left(\| \GRAD\psf \|^2_{L^2(\varpi,\Omega)} +  \| \GRAD \ysf \|^2_{L^\infty(\Omega)} \right) ,
\]
for some nonessential constant $c$.

We estimate the term $j'(\usf)(\Pi_{L^2} \usf -\usf)$ as follows:
\begin{align*}
  j'(\usf)(\Pi_{L^2} \usf -\usf) &=  (\psf + \lambda \usf, \Pi_{L^2}\usf-\usf)_{L^2(\Omega)}  = 
  (\psf + \lambda \usf - \Pi_{L^2}(\psf + \lambda \usf ), \Pi_{L^2}\usf-\usf)_{L^2(\Omega)} \\
  &\leq 
  \frac12 \| \Pi_{L^2} \usf - \usf \|_{L^2(\Omega)}^2 + \frac{1}{2}\| \psf - \Pi_{L^2} \psf \|_{L^2(\Omega)}^2 
  \\
  &\leq \frac12 \| \Pi_{L^2} \usf - \usf \|_{L^2(\Omega)}^2 + c \sigma_\T^2 \| \GRAD \psf \|_{L^2(\varpi,\Omega)}^2,
\end{align*}
for some nonessential constant $c$. We have used the properties of $\Pi_{L^2}$, together with the Sobolev-Poincar\'e inequality of \cite[Theorem 6.2]{NOS2}; see also \cite[Corollary 7.9]{NOS2}.

We now proceed to estimate the term $\textrm{III}$ in the proof of Lemma~\ref{lem:abserror} as follows:
\begin{align*}
\textrm{III} & := b(\Pi_{L^2} \usf - \usf, \psf_{h} - \wbpsf_{h}(\usf) ) = (\Pi_{L^2} \usf - \usf, \psf_{h} - \wbpsf_{h}(\usf)   )_{L^2(\Omega)}\\
&\leq \frac1{2\varepsilon} \| \Pi_{L^2} \usf - \usf \|_{L^2(\Omega)}^2
+ \frac\varepsilon2 \|  \psf_{h} - \wbpsf_{h}(\usf) - \Pi_{L^2}(  \psf_{h} - \wbpsf_{h}(\usf)  ) \|_{L^2(\Omega)}^2 \\
&\leq \frac1{2\varepsilon} \| \Pi_{L^2} \usf - \usf \|_{L^2(\Omega)}^2 + c \varepsilon \sigma_\T^2 \| \GRAD ( \psf_{h} - \wbpsf_{h}(\usf) ) \|_{L^2(\varpi,\Omega)}^2,
\end{align*}
where we have used the properties of $\Pi_{L^2}$ together with the Sobolev-Poincar\'e inequality of \cite[Theorem 6.2]{NOS2}. In the previous estimate $\varepsilon>0$ is arbitrary and $c$ is a nonessential constant; see Lemma \ref{lem:abserror}. Using now the fact that $\delta_z \in H^1_0(\varpi^{-1},\Omega)'$, and the discrete inf-sup stability of \cite[Lemma 7.8]{NOS2}, we have that
\begin{align}
\nonumber
\| \psf_{h} - \hat{\wbpsf}_{h} \|_{H^1_0(\varpi,\Omega)} &\lesssim \| \ysf_h - \hat{\wbysf}_h  \|_{L^\infty(\Omega)}\\
&\lesssim \fraki_\T \| \hat{\wbysf}_h - \ysf_h \|_{H^1_0(\Omega)} \leq \fraki_\T \| \usf - \usf_h \|_{L^2(\Omega)},
\label{eq:y_Linf2}
\end{align}
where $\fraki_\T$ is the mesh-dependent factor in the inverse inequality between $L^\infty(\Omega)$ and $H^1(\Omega)$ (see \cite[Lemma 4.9.2]{BrennerScott} for $n=2$ and \cite[Lemma 1.142]{Guermond-Ern} for $n = 3$):
\begin{equation}
\label{eq:i_T}
\fraki_\T =  ( 1 + |\log h_\T| )^\frac{1}{2} \quad \textrm{if} \quad n = 2, \quad \textrm{and}\quad \fraki_\T = h_\T^{-1/2} \quad \textrm{if} \quad n = 3.
\end{equation}
Therefore, we have derived 
\[
  \textrm{III} \leq \frac1{2\varepsilon}\| \Pi_{L^2} \usf - \usf \|_{L^2(\Omega)}^2 + c \varepsilon \sigma_\T^2 \fraki_\T^2 \| \usf - \usf_h \|_{L^2(\Omega)}^2.
\]
We now examine the product $\sigma_\T \fraki_\T$ for $h_{\T}$ sufficiently small:
\[
  \sigma_\T \fraki_\T = \begin{dcases}
                          h_\T |\log h_\T|^{\tfrac{3}{2}} & \textrm{if} \quad n = 2,\\
                          |\log h_\T| & \textrm{if} \quad n = 3,
                        \end{dcases}
\]
Collecting the derived estimates we arrive at the desired estimates \eqref{eq:tracking_estimate1}--\eqref{eq:tracking_estimate2} by considering a judicious choice of $\varepsilon$. To be precise, when $n = 2$, we set $\varepsilon = 1$. If $n=3$, we consider $\varepsilon = c |\log h_{\T}|^{-2}$.
\end{proof}

\begin{proposition}[regularity of $\usf$]
\label{prop:regup}
If $\usf$ solves \eqref{eq:defofJp}--\eqref{eq:cc2} then $\usf \in H^1(\varpi,\Omega)$. 
\end{proposition}
\begin{proof}
From \cite[Section 3.6.3]{Tbook}, $\usf$ solves \eqref{eq:VIp} if and only if
\[
  \usf = \max \left\{ \asf, \min\left\{ \bsf, -\frac1\lambda \psf\right\} \right\}.
\]
This immediately yields $\usf \in H^1(\varpi,\Omega)$ by invoking \cite[Theorem A.1]{KSbook}.
\end{proof}

Using this smoothness and an interpolation theorem between weighted spaces we can bound the projection error in Corollary~\ref{col:errp} and finish the error estimates \eqref{eq:tracking_estimate1}--\eqref{eq:tracking_estimate2} as follows.

\begin{theorem}[rates of convergence]
\label{th:tracking_error_estimates}
In the setting of Corollary~\ref{col:errp} we have
\begin{equation}
\label{eq:tracking_error_estimates_2d}
  \| \usf - \usf_h \|_{L^2(\Omega)} \lesssim  h_{\T} |\log h_{\T}|\left(\| \GRAD\psf \|_{L^2(\varpi,\Omega)} +  \| \GRAD \ysf \|_{L^\infty(\Omega)} \right)
\end{equation}
for $n=2$. If $n=3$, we have
\begin{equation}
\label{eq:tracking_error_estimates_3d}
  \| \usf - \usf_h \|_{L^2(\Omega)} \lesssim h_{\T}^{\frac{1}{2}} |\log h_{\T}|^2 \left(\| \GRAD\psf \|_{L^2(\varpi,\Omega)} +  \| \GRAD \ysf \|_{L^\infty(\Omega)} \right).
\end{equation}
The hidden constants in both estimates are independent of $\T$, the continuous and discrete optimal pairs.
\end{theorem}
\begin{proof}
We only need to bound the projection error $\| \usf - \Pi_{L^2} \usf\|_{L^2(\Omega)}$. Proposition~\ref{prop:regup} yields $\usf \in H^1(\varpi,\Omega)$, then, invoking \cite[Theorem 6.2]{NOS2}, we derive
\[
  \| \usf - \Pi_{L^2} \usf\|_{L^2(\Omega)} \lesssim \sigma_\T \| \GRAD \usf \|_{L^2(\varpi,\Omega)}.
\]
Substituting the previous estimate in the conclusion of Corollary~\ref{col:errp} we derive the claimed convergence rates.
\end{proof}

On the basis of the previous results we now derive an error estimate for the approximation of the state variable. 

\begin{theorem}[rates of convergence]
\label{th:tracking_error_estimates_state}
In the setting of Corollary~\ref{col:errp} we have
\begin{equation}
\label{eq:tracking_error_estimates_2d_state}
  \| \ysf - \ysf_h  \|_{L^{\infty}(\Omega)} \lesssim h_{\T} |\log h_{\T}| \left(\| \GRAD\psf \|_{L^2(\varpi,\Omega)} +  \| \GRAD \ysf \|_{L^\infty(\Omega)} \right),
\end{equation}
for $n=2$. If $n=3$, we have
\begin{equation}
\label{eq:tracking_error_estimates_3d_state}
\| \ysf - \ysf_h  \|_{L^{\infty}(\Omega)} \lesssim h_{\T}^{\frac{1}{2}} |\log h_{\T}|^2 \left(\| \GRAD\psf \|_{L^2(\varpi,\Omega)} +  \| \GRAD \ysf \|_{L^\infty(\Omega)} \right).
\end{equation}
The hidden constants in both estimates are independent of $\T$, and the continuous and discrete optimal pairs.
\end{theorem}
\begin{proof}
We start with a simple application of the triangle inequality:
\begin{equation}
\label{eq:pointwise_y_1}
 \|\ysf- \ysf_h   \|_{L^{\infty}(\Omega)} \leq \|\ysf - \hat{\wbysf} \|_{L^{\infty}(\Omega)} + \| \hat{\wbysf} -  \ysf_h   \|_{L^{\infty}(\Omega)},
\end{equation}
where $\hat{\wbysf}$ solves $a(\hat{\wbysf},v) = (\usf_h,v)$ for all $v \in H^1_0(\Omega)$. The second term on the right hand side of the previous inequality is controlled in view of standard pointwise estimates for finite elements. In fact, \cite[Theorem 5.1]{MR637283} yields
\begin{equation}
\label{eq:y_Linf_aux}
 \|\hat{\wbysf} - \ysf_h \|_{L^{\infty}(\Omega)} \lesssim h_\T |\log h_\T| \| \nabla \hat{\wbysf} \|_{L^\infty(\Omega)} .
\end{equation}
To control the first term on the right hand side of \eqref{eq:pointwise_y_1}, we invoke the results of \cite{JK:95} and conclude that there is $r>n$ such that
\[
  \| \ysf - \hat{\wbysf}   \|_{W^{1,r}(\Omega)} \lesssim \|  \usf -\usf_h \|_{L^2(\Omega)}.
\]
This, in conjunction with the embedding $W^{1,r}(\Omega) \hookrightarrow C(\bar \Omega)$ for $r > n$, yields
\begin{equation*}
\| \ysf - \hat{\wbysf}   \|_{L^{\infty}(\Omega)} \lesssim \|  \usf -\usf_h \|_{L^2(\Omega)}.
\end{equation*}
In view of \eqref{eq:y_Linf_aux}, the previous estimate and the results of Theorem \ref{th:tracking_error_estimates} allow us to derive the desired error estimates. 
\end{proof}

\section{Optimization with singular sources}
\label{sec:delta}
Let us remark that, since the formulation of the adjoint problem \eqref{eq:adjp} led to an elliptic problem with Dirac deltas on the right hand side, the problem with point sources on the state \eqref{eq:defofJd}--\eqref{eq:cc3} is, in a sense, dual to one with point observations \eqref{eq:defofJp}--\eqref{eq:cc2}. In the latter, the functional space for the adjoint variable is the one needed for the state variable in \eqref{eq:defofJd}--\eqref{eq:cc3}. The analysis will follow 
the one presented in Section \ref{sub:discrp}. It is important to comment that problem \eqref{eq:defofJd}--\eqref{eq:cc3} has been analyzed before. We refer the reader to  \cite{MR3225501} for the elliptic case and to \cite{MR2983016, MR3150173, MR3116646,MR2995178} for the parabolic one. It is our desire in this section to show how the theory of Muckenhoupt weights can be used to analyze and approximate problem \eqref{eq:defofJd}--\eqref{eq:cc3}. In doing this, it will be essential to assume that $\textrm{dist}(\calD,\partial \Omega) \geq d_{\calD}>0$.  Set
\begin{enumerate}[$\bullet$]
  \item $\bbH = L^2(\Omega)$ and $C = \id$.
  
  \item $\bbU = \R^l$.
  
  \item $\bbY_1 = H^1_0(\varpi,\Omega)$ and $\bbX_1 = H^1_0(\varpi^{-1},\Omega)$, with $\varpi$ defined, as in Section~\ref{sub:analysisp}, by \eqref{eq:defofoweight}.
  \item $\bbY_2 = \bbX_2 = H^1_0(\Omega)$ and
  \[
    a(v,w) = \int_\Omega \GRAD v(x)^\intercal \GRAD w(x) \diff x.
  \]

  \item The bilinear form $b: \bbU \times (\bbX_1 + \bbX_2)$ is 
  \[
    b(\mathbf{v},w) = \sum_{z \in \calD} v_z \langle \delta_z , w\rangle_{H^1_0(\varpi^{-1},\Omega)' \times H^1_0(\varpi^{-1},\Omega)}.
  \]
  Since, for $z \in \Omega$, $\delta_z \in H^1_0(\varpi^{-1},\Omega)'$, we have that $b$ is continuous on $\mathbb{R}^l \times H^1_0(\varpi^{-1},\Omega)$. 
\end{enumerate}

For $\asf,\bsf \in \R^l$ with $\asf_z < \bsf_z$ we define the set of admissible controls as
\[
  \bbU_\delta = \left\{ \mathbf{u} \in \R^l: \asf_z \leq u_z \leq \bsf_z, \ \forall z \in \calD \right\}.
\]
The space of controls is already discrete, so we set $\bbU^h = \bbU$ and $\Uad^h = \bbU_\delta$. Finally we set, for $i=1,2$, $\bbX_i^h=\bbY_i^h = \V(\T)$. Since the bilinear form $b$ is not continuous on $\mathbb{U} \times \mathbb{X}_2$, we need to slightly modify the arguments of Lemma \ref{lem:abserror}. In what follows, for $v \in C(\bar \Omega)$ and $w \in \R^{l}$ we define
\begin{equation}
\label{eq:innerD}
  \langle v,w \rangle_{\calD}:= \sum_{z \in \mathcal{D}} v(z) w_z.
\end{equation}

In this setting, the main error estimate for problem \eqref{eq:defofJd}--\eqref{eq:cc3} is provided below. We comment that our proof is inspired in the arguments developed in \cite[Theorem~3.7]{MR3225501} and \cite{MR3116646,MR645661}.

\begin{theorem}[rates of convergence]
\label{th:errd}
Let $\bu$ and $\bu_h$ be the continuous and discrete optimal controls, respectively, and assume that for every $q \in (2,\infty)$, $y_d \in L^q(\Omega)$. Let $\epsilon>0$ and $\Omega_1$ be such that $\calD \Subset \Omega_1 \Subset \Omega$. If $n=2$, then
\begin{equation}
\label{eq:delta_2d}
   \| \bu - \bu_h \|_{\R^{l}} \lesssim h_\T^{2-\epsilon} \left( \| \bu \|_{\R^l} + \| \psf \|_{H^2(\Omega)} + \| \psf \|_{W^{2,r}(\Omega_1)} \right).
\end{equation} 
On the other hand, if $n = 3$, then
\begin{equation}
\label{eq:delta_3d}
  \| \bu - \bu_h \|_{\R^{l}} \lesssim h_\T^{1-\epsilon} \left( \| \bu \|_{\R^l} + \| \psf \|_{H^2(\Omega)} +\| \psf \|_{W^{2,r}(\Omega_1)} \right),
\end{equation}
where $r < n/(n-2)$. The hidden constants in both estimates are independent of $\T$, and the continuous and discrete optimal pairs.
\end{theorem}
\begin{proof}
We start the proof by noticing that, since $\ysf -y_d \in L^2(\Omega)$ and $\Omega$ is convex, standard regularity arguments \cite{Grisvard} yield $\psf \in H^2(\Omega)  \hookrightarrow C(\bar\Omega)$ . This guarantees that pointwise evaluations of $\psf$ are well defined. Moreover, since, in this setting, $\Uad^h = \Uad$ estimate \eqref{eq:aux} reduces to
\[
\lambda \| \bu - \bu_h \|^2_{\R^{l}} \leq \langle \psf - \psf_h,  \bu_h - \bu \rangle_{\mathcal{D}},
\]
where $\langle \cdot ,  \cdot \rangle_{\mathcal{D}}$ is defined in \eqref{eq:innerD}. Adding and subtracting the solution to \eqref{eq:hatph} $\hat{\wbpsf}_h$, we obtain that
\begin{equation}
\label{eq:bu-bu_h}
\lambda \| \bu - \bu_h \|^2_{\R^l} \leq \langle\psf - \hat{\wbpsf}_h,  \bu_h - \bu \rangle_{\mathcal{D}} + \langle\hat{\wbpsf}_h - \psf_h,  \bu_h - \bu \rangle_{\mathcal{D}}.
 \end{equation}
This, in view of $\langle\hat{\wbpsf}_h - \psf_h,  \bu_h - \bu \rangle_{\mathcal{D}}  = - \|\hat{\wbysf}_h - \ysf_h   \|_{L^2(\Omega)}^2$, implies that
\begin{align}
\nonumber
\lambda \| \bu - \bu_h \|_{\R^{l}}^2 + \|\hat{\wbysf}_h - \ysf_h   \|_{L^2(\Omega)}^2 & \leq \langle \psf - \hat{\wbpsf}_h,  \bu_h - \bu \rangle_{\mathcal{D}}
\\ 
\label{eq:bu-bu_h2}
& \leq \langle \psf  - \qsf_h,  \bu_h - \bu \rangle_{\mathcal{D}} + \langle \qsf_h - \hat{\wbpsf}_h,  \bu_h - \bu \rangle_{\mathcal{D}},
\end{align}
where $\qsf_h$ is defined as the unique solution to
\[
  \qsf_h \in \V(\T): \quad a(w_h,\qsf_h ) = (\ysf - y_d,w_h)_{L^2(\Omega)} \quad \forall w_h \in \V(\T).
\]

Since, by assumption, we have that $d_\calD >0$ we can conclude that there are smooth subdomains $\Omega_0$ and $\Omega_1$ such that $ \calD \subset \Omega_0 \Subset \Omega_1 \Subset \Omega$. In view of \eqref{eq:bu-bu_h2}, this key property will allow us to derive interior $L^{\infty}$-estimates for $\psf - \qsf_h$ and $\qsf_h - \hat{\wbpsf}_h$.

Let us first bound the difference $\psf - \qsf_h$. To do this, we notice that, since $\ysf \in W_0^{1,s}(\Omega)$ for $s<n/(n-1)$, a standard Sobolev embedding result implies that $\ysf \in L^r(\Omega)$ with $r \leq ns/(n-s) <n/(n-2)$. Then, on the basis of the fact that $y_d \in L^{q}(\Omega)$  for $q < \infty$, interior regularity results guarantee that $\psf \in W^{2,r}(\Omega_1)$ for $r < n/(n-2)$. Consequently, since $\qsf_h$ corresponds to the Galerkin approximation of $\psf$, \cite[Theorem 5.1]{MR0431753} yields, when $n=2$, that for any $\epsilon>0$ we have
\begin{equation}
\label{eq:psf-qsf2d}
  \|\psf - \qsf_h \|_{L^{\infty}(\Omega_0)} \lesssim  
  \left( h_{\T}^{2-\epsilon} \| \psf \|_{W^{2,r}(\Omega_1)} + h_{\T}^2 \| \psf \|_{H^2(\Omega)} \right).
\end{equation}
When $n=3$, we have that $\psf \in H^1_0(\Omega) \cap W^{2,r}(\Omega_1)$ for $r<3$ and, as a consequence,
\begin{equation}
\label{eq:psf-qsf3d}
  \|\psf - \qsf_h \|_{L^{\infty}(\Omega_0)} \lesssim  
  \left( h_{\T}^{1-\epsilon} \| \psf \|_{W^{2,r}(\Omega_1)} + h_{\T}^2 \| \psf \|_{H^2(\Omega)} \right).
\end{equation}

It remains then to estimate the difference $P_h= \qsf_h - \hat{\wbpsf}_h$. To do so we employ a duality argument that combines the ideas of \cite[Corollary 7.9]{NOS2} and \cite{MR3116646,MR645661}. We start by defining $\varphi \in H^1_0(\Omega)$ as the solution to
\begin{equation}
\label{eq:a_varphi}
  a(v,\varphi) = \int_\Omega \sgn( \ysf - \hat{\wbysf}_h ) v \quad \forall v \in H^1_0(\Omega),
\end{equation}
where $\hat{\wbysf}_h$ solves \eqref{eq:hatyh}.
Notice that $\| \sgn( \ysf - \hat{\wbysf}_h ) \|_{L^\infty(\Omega)} \leq 1$ for all $\T \in \Tr$. Therefore, \cite[Theorem 5.1]{MR0431753} followed by \cite[Theorem 5.1]{MR637283} lead to (see also \cite[Lemma 3.2]{MR3225501})
\begin{equation}
\label{eq:Schatz}
  \| \varphi - \varphi_h \|_{L^\infty(\Omega_0)} \lesssim h_\T^{2} |\log h_{\T}|^2
\end{equation}
where $\varphi_h$ is the Galerkin projection of $\varphi$ and the hidden constant does not depend on $\T$ nor $\varphi$.
In addition we have that $\varphi \in H^2(\Omega)\cap H^1_0(\Omega) \hookrightarrow H^1_0(\varpi^{-1},\Omega)$ \cite[Lemma 7.6]{NOS2}. Therefore $\varphi$ is a valid test function in the variational problem that $\ysf$ solves. Then, using the continuity of the bilinear form $a$ and Galerkin orthogonality we arrive at
\begin{align*}
  \| \ysf - \hat{\wbysf}_h \|_{L^1(\Omega)} &= \int_\Omega \sgn( \ysf - \hat{\wbysf}_h ) ( \ysf - \hat{\wbysf}_h ) = a(\ysf - \hat{\wbysf}_h, \varphi ) \\
    & = a( \ysf, \varphi - \varphi_h) = \langle \varphi - \varphi_h, \bu\rangle_\calD
    \lesssim \| \bu \|_{\R^l} \| \varphi - \varphi_h \|_{L^\infty(\Omega_0)} \\
    &\lesssim h_\T^2 |\log h_{\T}|^2 \| \bu \|_{\R^l},
\end{align*}
where in the last step we used estimate \eqref{eq:Schatz}.

We now recall that $P_h$ solves
\[
  a(w_h,P_h) = (\ysf - \hat{\wbysf}_h,w_h)_{L^2(\Omega)} \quad \forall w_h \in \V(\T),
\]
an inverse inequality and a stability estimate for the problem above yield
\[
  \| P_h \|_{L^\infty(\Omega)}^2 \lesssim \fraki_\T^2 \| \GRAD P_h \|_{L^2(\Omega)}^2 \leq \fraki_\T^2 
  \| \ysf - \hat{\wbysf}_h \|_{L^1(\Omega)} \| P_h \|_{L^\infty(\Omega)},
\]
where $\fraki_\T$ is defined in \eqref{eq:i_T}. In conclusion,
\begin{equation}
\label{eq:Ph}
  \| P_h \|_{L^\infty(\Omega)} \lesssim\fraki_\T^2 \| \ysf - \hat{\wbysf}_h \|_{L^1(\Omega)} \lesssim  \fraki_\T^2 h_\T^2 |\log h_\T|^2 \| \bu \|_{\R^{l}}.
\end{equation}

Combining the obtained pointwise bounds for $\psf - \qsf_h$ and $\qsf_h - \hat{\wbpsf}_h$ we obtain the desired estimates.
\end{proof}

\begin{remark}[comparison with the literature]\rm
\label{rm:comparison}
Reference \cite{MR3225501} claims to obtain better rates than those in Theorem \ref{th:errd}, namely they can trade the term $h_\T^{-\epsilon}$ by a logarithmic factor $|\log h_\T|^s$ with $s \geq1$ but small. However, when following the arguments that lead to this estimate (see \cite[formula (3.40)]{MR3225501}) one realizes that a slight inaccuracy takes place. Namely, the authors claim that, for $s<n/(n-1)$,
\[
 h_{\T}^{3-n/s}|\log h_{\T}| \lesssim h_{\T}^2 |\log h_{\T}|.
\]
However, $3-n/s<4-n$ which, for $n=2$ or $n=3$ reduces to the estimates that we obtained in Theorem~\ref{th:errd}. 
\end{remark}

To conclude, we present an error estimate for the state variable.

\begin{corollary}[rates of convergence]
In the setting of Theorem~\ref{th:errd} we have
\[
 \|\ysf- \ysf_h   \|_{L^2(\Omega)} \lesssim \sigma_\T 
 \left( \| \bu \|_{\R^l} + \| \psf \|_{H^2(\Omega)} +\| \psf \|_{W^{2,r}(\Omega_1)} + \| \GRAD \ysf \|_{L^2(\varpi,\Omega)} \right).
\]
The hidden constant is independent of $\T$ and the continuous and discrete optimal pairs.
\end{corollary}
\begin{proof}
A simple application of the triangle inequality yields
\begin{equation}
\label{eq:pointwise_y_2}
 \|\ysf- \ysf_h   \|_{L^2(\Omega)} \leq \|\ysf - \hat{\wbysf}_h \|_{L^2(\Omega)} + \| \hat{\wbysf}_h -  \ysf_h   \|_{L^2(\Omega)},
\end{equation}
where $\hat{\wbysf}_h$ solves \eqref{eq:hatyh}. To estimate the first term on the right-hand-side of the previous expression, we invoke \cite[Corollary 7.9]{NOS2} and arrive at
\[
  \| \ysf - \hat{\wbysf}_h \|_{L^2(\Omega)} \lesssim \sigma_\T \| \GRAD \ysf \|_{L^2(\varpi,\Omega)} . 
\]
Using \eqref{eq:bu-bu_h2} and the results of Theorem \ref{th:errd} we bound the second term on the right hand side of \eqref{eq:pointwise_y_2}. This concludes the proof.
\end{proof}

\section{Numerical Experiments}
\label{sec:NumExp}
In this section we conduct a series of numerical experiments that illustrate the performance of scheme \eqref{eq:Jdis}--\eqref{eq:ccdis} when is used to approximate the solution to the optimization problem with point observations studied in Section~\ref{sec:points} and the one with singular sources analyzed in Section~\ref{sec:delta}.

\subsection{Implementation}
All the numerical experiments that will be presented have been carried out with the help of a code that is implemented using \texttt{C++}. The matrices involved in the computations have been assembled exactly, while the right hand sides and the approximation error are computed by a quadrature formula which is exact for polynomials of degree 19 for two dimensional domains and degree 14 for three dimensional domains. The corresponding linear systems are solved using the multifrontal massively parallel sparse direct solver (MUMPS) \cite{MUMPS1,NUMPS2}. To solve the minimization problem \eqref{eq:Jdis}--\eqref{eq:ccdis} we use a Newton-type primal-dual active set strategy \cite[Section 2.12.4]{Tbook}.

For all our numerical examples we consider $\lambda=1$ and construct exact solutions based on the fundamental solutions for the Laplace operator:
\begin{equation}
\label{exact_adjoint}
\phi(x) = \begin{dcases}
                  -\frac{1}{2\pi}\sum_{z \in \frakS}\log|x-z|, & \textrm{if}~\Omega\subset\mathbb{R}^{2},\\
                  \frac{1}{4\pi}\sum_{z \in \frakS}\frac{1}{|x-z|}, & \textrm{if}~\Omega\subset\mathbb{R}^{3},
                \end{dcases}
\end{equation}
where, depending on the problem, $\frakS = \calZ$ or $\frakS = \calD$.

We must remark that the introduction of weights is only to simplify the analysis and that these are never used in the implementation. This greatly simplifies it and allows for the use of existing codes.

\subsection{Optimization with point observations: one point} We set $n=2$ and $\Omega=(0,1)^2$. We consider the control bounds that define the set $\mathbb{U}_{\mathcal{Z}}$ as $\asf=-0.4$ and $\bsf=-0.2$. To construct an exact solution to the optimization problem with point observations, we slightly modify the corresponding state equation by adding a forcing term $\fsf \in L^2(\Omega)$, \ie we replace \eqref{eq:statep} by the following problem:
\begin{equation}
\label{eq:statep_new}
\ysf \in H_0^1(\Omega): \qquad a(\ysf,w) = (\usf + \fsf,w) \quad \forall w \in H_0^1(\Omega). 
\end{equation}
We then define the exact optimal state, the observation set and the desired point value as follows:
\begin{equation*}
  \ysf(x_{1},x_{2}) = 32x_1x_2(1-x_1)(1-x_2), \qquad
  \calZ=\{(0.5,0.5)\}, \qquad \ysf_{(0.5,0.5)}=1.
\end{equation*}
The exact optimal adjoint state is given by \eqref{exact_adjoint} and the right hand side $\fsf$ is computed accordingly.

To present the performance of the fully discrete scheme \eqref{eq:Jdis}--\eqref{eq:ccdis}, we consider a family of quasi--uniform meshes $\{ \T_k \}_{k=1}^8$. We set $N(k) = \# \T_k$, that is, the total number of degrees of freedom. In addition, we denote by $\mathrm{EOC}_{\qsf}(k)$ the corresponding experimental order of convergence associated to the variable $\qsf$, which is computed using
\[
 \mathrm{EOC}(k) = \frac{\ln \left(e_{\qsf}(k-1)/e_{\qsf}(k) \right)}{\ln \left( N(k-1)/ N(k) \right)},
\]
where $e_{\qsf}(k)$ denotes the resulting error in the approximation of the control variable $\qsf$ and $k \in \{2,\cdots,8\}$.
\begin{table}[!h]
  \begin{center}
    \begin{tabular}{||c||c||c||c||c||}
      \hline
      DOFs          &  $\| \usf - \usf_{\T_k}\|_{L^2(\Omega)}$ & $\mathrm{EOC}_{\usf}$ & $\| \ysf - \ysf_{\T_k}\|_{L^{\infty}(\Omega)}$ & $\mathrm{EOC}_{\ysf}$\\
      \hline
      42  &  0.0456202    & --    &  0.3940558  & -- \\
      \hline
      146  & 0.0259039   & -0.4542396 &   0.1220998   & -0.9403796 \\
      \hline
      546  & 0.0106388  & -0.6746618   &  0.0356279& -0.9338121  \\
      \hline
      2114 & 0.0053128  & -0.5129453    &  0.0104755 & -0.9042427 \\
      \hline
      8322 & 0.0026798 & -0.4994327      & 0.0030256 & -0.9063059\\
      \hline
      33026  & 0.0013372 & -0.5043272     & 0.0008921 & -0.8860222\\
      \hline
      131586  & 0.0006675 & -0.5025385  &  0.0002586 & -0.8957802\\
      \hline
      525314  & 0.0003340 & -0.5000704  &  7.359881e-05 & -0.9077666\\
      \hline
    \end{tabular} 
  \end{center}
\caption{Experimental order of convergence of scheme \eqref{eq:Jdis}--\eqref{eq:ccdis} when is used to approximate the solution to the optimization problem of Section~\ref{sec:points} with one observation point. The $\mathrm{EOC}_{\usf}$ is in agreement with estimate \eqref{eq:tracking_error_estimates_2d} of Theorem~\ref{th:tracking_error_estimates}: the family $\{\T_{k} \}$ is quasi-uniform and, thus, $h_{\T_k} \approx N(k)^{-1/2}$, which is what we observe. The $\mathrm{EOC}_{\ysf}$ reveals a quadratic order and illustrates that our error estimate \eqref{eq:tracking_error_estimates_2d_state} might be pessimistic.}
\label{table:example1_tracking}
\end{table}

Table~\ref{table:example1_tracking} shows that, when approximating the optimal control variable, the $\mathrm{EOC}_{\usf}$ is in agreement with the estimate \eqref{eq:tracking_error_estimates_2d}. This illustrates the sharpness of the derived estimate up to a logarithmic term. We comment that, since the family $\{ \T_k \}_{k=1}^8$ is quasi-uniform, we then have that $h_{\T_{k}} \approx N(k)^{-1/2}$. Consequently, \eqref{eq:tracking_error_estimates_2d} reads as follows:
\begin{equation}
\label{eq:error_DOFs}
 \| \usf - \usf_{\T_k}\|_{L^2(\Omega)} \lesssim N(k)^{-\tfrac{1}{2}} |\log N(k)|.
\end{equation}
Table \ref{table:example1_tracking} also presents the $\mathrm{EOC}_{\ysf}$ obtained for the approximation of the optimal state variable $\ysf$: $h_{\T_k}^2 \approx N(k)^{-1}$. These results show that the derived error estimate \eqref{eq:tracking_error_estimates_2d_state} might be pessimistic.

\subsection{Optimization with point observations: four points} The objective of this numerical experiment is to test the performance of the fully discrete scheme \eqref{eq:Jdis}--\eqref{eq:ccdis} when more observation points are considered.

Let us consider $n=2$ and $\Omega = (0,1)^2$. The control bounds defining the set $\mathbb{U}_{\mathcal{Z}}$ are given by $\asf = -1.2$ and $\bsf = -0.7$. The state equation \eqref{eq:statep} is replaced by \eqref{eq:statep_new}, which allows the incorporation of a forcing term $\fsf$. We set
\[
 \calZ=\{(0.75,0.75),(0.75,0.25),(0.25,0.75),(0.25,0.25)\},
\]
with corresponding desired values
\[
 \ysfd_{(0.75,0.75)}=1,\qquad \ysfd_{(0.25,0.25)}=1,\qquad \ysfd_{(0.75,0.25)}=0.5, \qquad \ysfd_{(0.25,0.75)}=0.5.
\]
The exact optimal state variable is then given by
\[
 \ysf(x_{1},x_{2})=2.75-2x_1-2x_2+4x_1x_2
\]
and the exact optimal adjoint state is given by \eqref{exact_adjoint}.

\begin{table}[!h]
  \begin{center}
    \begin{tabular}{||c||c||c||c||c||}
      \hline
      DOFs          &  $\| \usf - \usf_{\T_k}\|_{L^2(\Omega)}$ & EOC &  $\| \ysf - \ysf_{\T_k} \|_{L^{\infty}(\Omega)}$  & EOC \\
      \hline
      42  &  0.0285416   & --         & 0.0595256 & -- \\
      \hline
      146  & 0.0285084  & -0.0009357 & 0.0152388 &  -1.0936039   \\
      \hline
      546  & 0.0208153  & -0.2384441 &  0.0039226  & -1.0288683 \\
      \hline
      2114 & 0.0116163  & -0.4308717 &  0.0010313  & -0.9868631  \\
      \hline
      8322 & 0.0061821 & -0.4602926 &  0.0002708 & -0.9758262    \\
      \hline
      33026  & 0.0030792 & -0.5056447   &  7.057710e-05 & -0.9755383 \\
      \hline 
      131586  & 0.0014908 & -0.5247299  & 1.729492e-05  & -1.0173090\\
      \hline
      525314  & 0.0007618 & -0.4849766  & 4.503108e-06 & -0.9720511\\
      \hline
    \end{tabular} 
  \end{center}
\caption{Experimental order of convergence of scheme \eqref{eq:Jdis}--\eqref{eq:ccdis} when is used to approximate the solution of the problem of Section~\ref{sec:points} with four observation points. The $\mathrm{EOC}_{\usf}$ is in agreement with estimate \eqref{eq:tracking_error_estimates_2d} of Theorem~\ref{th:tracking_error_estimates}: the family $\{\T_{k} \}$ is quasi-uniform, so that $h_{\T_k} \approx N(k)^{-1/2}$, which is what we observe. The $\mathrm{EOC}_{\ysf}$ reveals a quadratic order and illustrates that our error estimate \eqref{eq:tracking_error_estimates_2d_state} might be pessimistic.}
\label{table:example2_tracking}
\end{table}

Table~\ref{table:example2_tracking} shows that the $\mathrm{EOC}_{\usf}$ is in agreement with estimate \eqref{eq:tracking_error_estimates_2d} of Theorem~\ref{th:tracking_error_estimates}. This illustrates the robustness of scheme \eqref{eq:Jdis}--\eqref{eq:ccdis} when more observations points are considered.

\subsection{Optimization with point observations: a three dimensional example} We set $n=3$ and $\Omega = (0,1)^3$. We define $\asf=-15$ and $\bsf=-5$. The optimal state is
\[
  \ysf(x_{1},x_{2},x_{3})=\frac{8192}{27}x_1 x_2 x_3(1-x_1)(1-x_2)(1-x_3),
\]
whereas the optimal adjoint state is defined by \eqref{exact_adjoint}. The set of observation points is 
\begin{align*}
  \calZ = &\left\{(0.25,0.25,0.25),(0.25,0.25,0.75),(0.25,0.75,0.25),(0.25,0.75,0.75), \right. \\
  &\left. (0.75,0.25,0.25),(0.75,0.25,0.75),(0.75,0.75,0.25),(0.75,0.75,0.25) \right\}
\end{align*}
and we set $\ysf_z=1$ for all $z \in \calZ$.

\begin{table}[!h]
  \begin{center}
    \begin{tabular}{||c||c||c||}
      \hline
      DOFs          &  $\| \usf - \usf_{\T_k}\|_{L^2(\Omega)}$   & $\mathrm{EOC}_{\usf}$ 
\\
      \hline
      1419  &  0.0274726   & --     \\
      \hline
      3694 & 0.0199406 & -0.3349167  \\
      \hline
      9976  & 0.0120137  & -0.5100352  \\
      \hline
      27800 & 0.0088690  & -0.2961201 \\
      \hline
      79645 & 0.0067903 & -0.2537367  \\
      \hline
      234683  &  0.0049961 & -0.2839348  \\
      \hline 
      704774  & 0.0036908 & -0.2753755 \\
      \hline
      2155291  & 0.0026540 & -0.2950207\\
      \hline
    \end{tabular} 
  \end{center}
\caption{Experimental order of convergence of scheme \eqref{eq:Jdis}--\eqref{eq:ccdis} when used to approximate the solution to the optimization problem of Section~\ref{sec:points} in a three dimensional example. The $\mathrm{EOC}_{\usf}$ is in agreement with estimate \eqref{eq:tracking_error_estimates_3d} of Theorem~\ref{th:tracking_error_estimates}: the family $\{\T_{k} \}$ is quasi-uniform and then $h_{\T_k} \approx N(k)^{-1/3}$.}
\label{table:example3_tracking}
\end{table}
Table~\ref{table:example3_tracking} illustrates that $\mathrm{EOC}_{\usf}$ is in agreement with estimate \eqref{eq:tracking_error_estimates_3d} of Theorem \ref{th:tracking_error_estimates} due to the fact that the family $\{\T_{k} \}$ is quasi-uniform and then $h_{\T_k} \approx \mathrm{DOFs}(k)^{-1/3}$.

\subsection{Optimization with singular sources} 
We now explore the performance of scheme \eqref{eq:Jdis}--\eqref{eq:ccdis} when is used to solve the optimization problem with singular sources. We set $n=2$ and $\Omega = (0,1)^2$. We consider $\mathcal{D} = (0.5, 0.5)$ and the control bounds that define the set $\mathbb{U}_{\delta}$ are $\asf = 0.3$ and $\bsf = 0.7$. The desired state and the exact adjoint state correspond to
\[
 \psf(x_1,x_2) = -32x_1x_2(1-x_1)(1-x_2), \qquad \ysf_d = -\sin(2\pi x)\cos(2\pi x)
\]
The exact optimal state is given by \eqref{exact_adjoint}.

\begin{table}[!h]
  \begin{center}
    \begin{tabular}{||c||c||c||}
      \hline
      DOFs   &  $\| \usf - \usf_{\T_k}\|_{L^2(\Omega)}$ & EOC \\
      \hline
      30  &  0.0940682   & --  \\       
      \hline
      86  & 0.0536485  & -0.5332256 \\
      \hline
      294  & 0.0207101  & -0.7743303 \\
      \hline
      1094 & 0.0068950  & -0.8369949\\
      \hline
      4230 & 0.0021408 & -0.8648701   \\
      \hline
      16646  & 0.0006380 & -0.8836678   \\
      \hline 
      66054 & 0.0001850 & -0.8981934  \\
      \hline
      263174 & 5.259841e-05 & -0.9098104\\
      \hline
       1050630 & 1.472536e-05 & -0.9196613 \\
      \hline
    \end{tabular} 
  \end{center}
\caption{Experimental order of convergence of scheme \eqref{eq:Jdis}--\eqref{eq:ccdis} when used to approximate the solution to the optimization problem with point sources of Section \ref{sec:delta}. The $\mathrm{EOC}_{\usf}$ reveals a quadratic order and illustrates our error estimate \eqref{eq:delta_2d}. }
\label{table:example1_sources}
\end{table}

\section*{Acknowledgements}
The authors would like to thank Johnny Guzm\'an for fruitful discussions regarding pointwise estimates and the regularity of elliptic problems in convex, polytopal domains. 
We would like to also thank Alejandro Allendes for his technical support.
\bibliographystyle{plain}
\bibliography{biblio}

\end{document}